\documentclass[10pt]{amsart}
\usepackage{amsthm, amsfonts, amssymb, amsmath, amscd, enumitem}
\usepackage[all]{xy}
\usepackage[margin=3.7cm]{geometry}

\usepackage{tikz}
\usetikzlibrary{calc}

\usepackage{color}
\usepackage[normalem]{ulem}

\newcommand{\C}{\mathbb{C}}
\newcommand{\R}{\mathbb{R}}

\newcommand{\G}{\Gamma}

\newcommand{\e}{\epsilon}

\newcommand{\bd}{\partial}

\newcommand{\Hyp}{\mathbb{H}}
\newcommand{\im}{\mathrm{Im}\:}

\newtheorem{thm}{Theorem}[section]
\newtheorem{lem}[thm]{Lemma}
\newtheorem{cor}[thm]{Corollary}
\newtheorem{prop}[thm]{Proposition}

\theoremstyle{definition}
\newtheorem{defn}[thm]{Definition}
\newtheorem{rem}[thm]{Remark}

\newtheorem*{ack}{Acknowledgments}

\begin{document}
\title{Strong hyperbolicity}
\author{Bogdan Nica}
\address[B.N.]{Mathematisches Institut, Georg-August Universit\"at G\"ottingen, Germany}
\email{bogdan.nica@gmail.com}
\author{J\'an \v{S}pakula}
\address[J.S.]{Mathematical Sciences, University of Southampton, UK}
\email{jan.spakula@soton.ac.uk}
\date{\today}

\begin{abstract}
We propose the metric notion of strong hyperbolicity as a way of obtaining hyperbolicity with sharp additional properties. Specifically, strongly hyperbolic spaces are Gromov hyperbolic spaces that are metrically well-behaved at infinity, and, under weak geodesic assumptions, they are strongly bolic as well. We show that $\mathrm{CAT}(-1)$ spaces are strongly hyperbolic. On the way, we determine the best constant of hyperbolicity for the standard hyperbolic plane $\Hyp^2$. We also show that the Green metric defined by a random walk on a hyperbolic group is strongly hyperbolic. A measure-theoretic consequence at the boundary is that the harmonic measure defined by a random walk is a visual Hausdorff measure.\end{abstract}

\maketitle

\section{Introduction}
Hyperbolicity for metric spaces, as introduced by Gromov \cite{Gro}, is a coarse notion of negative curvature which yields a very satisfactory theory. For certain analytic purposes, however, hyperbolicity by itself is not enough, and one needs certain enhancements.

After the groundbreaking work of Lafforgue \cite{Laf}, the resolution of the Baum - Connes conjecture for hyperbolic groups hinged on showing the following: every hyperbolic group $\G$ admits a geometric action on a roughly geodesic, strongly bolic hyperbolic space. Here the enhancement is \emph{strong bolicity} (see Section~\ref{sec: super}), for which one is willing to pay the small price of weakening the usual geodesic assumption. This is achieved by Mineyev and Yu in \cite{MY}, and the desired space is $\G$ itself endowed with a new metric. 

The enhancement needed in \cite{Nic} is a boundary metric structure with sharp properties. Let us say that a hyperbolic space $X$ is a \emph{good hyperbolic space} if the following hold: the Gromov product $(\cdot,\cdot)_o$ extends continuously to the bordification $X\cup \bd X$ for each basepoint $o\in X$, and there is some $\e>0$ such that $\exp(-\e\:(\cdot,\cdot)_o)$ is a metric on the boundary $\bd X$, again for each basepoint $o\in X$. For example, $\mathrm{CAT}(-1)$ spaces are good hyperbolic spaces, and one can take $\e=1$ (Bourdon \cite{Bou}). Now the construction of \cite{Nic} can be summarized as follows: from a geometric action of a hyperbolic group $\G$ on a roughly geodesic, good hyperbolic space $X$, one obtains a proper isometric action of $\G$ on an $L^p$-space associated to the double boundary $\bd X\times \bd X$. The concrete $X$ used in \cite{Nic} is $\G$ itself, equipped with a metric constructed by Mineyev \cite{Min1,Min2} and which is a slightly improved version of the metric used in \cite{MY}.

In this paper we introduce a certain metric condition that we call \emph{strong hyperbolicity}. The definition is given in Section~\ref{sec: super} and we will not repeat here. The main point about strong hyperbolicity is that there are useful consequences flowing out and important examples flowing in. 

As far as consequences are concerned, the following is a fairly easy outcome of the definition.

\begin{thm}\label{thm: A}
A strongly hyperbolic space is a good hyperbolic space. A roughly geodesic, strongly hyperbolic space is strongly bolic. 
\end{thm}

As for examples, we show the following.
\begin{thm}\label{thm: C}
$\mathrm{CAT}(-1)$ spaces are strongly hyperbolic.
\end{thm}

\begin{thm}\label{thm: G}
The Green metric arising from a random walk on a hyperbolic group is strongly hyperbolic.
\end{thm}

Theorem~\ref{thm: C} implies Bourdon's result mentioned above, but also one of the main results of Foertsch and Schroeder from \cite{FS}. From Theorem~\ref{thm: C} and its proof, we also deduce that the hyperbolic plane $\Hyp^2$ is $\log 2$-hyperbolic, in the sense of Gromov's original definition, and this is optimal (Corollary~\ref{cor: optimal}). Quite surprisingly, this was not known before. 

Good hyperbolicity for the Green metric on a hyperbolic group answers positively a question raised in \cite[Remark 19]{Nic}. As a consequence, we find that, on the boundary, the harmonic measure defined by a random walk coincides with the normalized Hausdorff measure defined by a Green visual metric (Corollary~\ref{cor: two equal measures}). Strong bolicity of the Green metric was recently proved by Ha\"issinsky and Mathieu in \cite{HM} by a different argument. The upshot is that the Green metric provides a more natural, and much simpler alternative to the hat metric constructed by Mineyev and Yu in \cite{MY}. Strong hyperbolicity for the Green metric, established herein, makes a similar point relative to Mineyev's improved hat metric from \cite{Min1,Min2}.

\section{Hyperbolic spaces - An elliptical reminder} 
\subsection{Hyperbolic spaces and their boundaries} We begin by recalling some basic notions and facts.

\begin{defn}[Gromov]\label{defn: gromov hyperbolic}
A metric space $X$ is \emph{$\delta$-hyperbolic}, where $\delta\geq 0$, if
\begin{align*}
	(x,y)_o \geq \min\{(x,z)_o,(y,z)_o\} - \delta
\end{align*}
for all $x,y,z,o\in X$; equivalently, the four-point condition
\begin{align*}
|x,y|+|z,t|\leq \max\big\{|x,z|+|y,t|, |z,y|+|x,t|\big\}+2\delta
\end{align*}
holds for $x,y,z,t\in X$.
\end{defn}

We write $|x,y|$ for the distance between two points $x,y\in X$. Recall that the \emph{Gromov product} is defined by the formula 
\begin{align*}
(x,y)_o = \tfrac{1}{2}\big(|x,o|+|y,o|-|x,y|\big).
\end{align*}

Let $X$ be a hyperbolic space, and fix a basepoint $o\in X$. A sequence $(x_i)$ in $X$ \emph{converges to infinity} if $(x_i,x_j)_o\to\infty$ as $i,j\to\infty$. Two sequences  $(x_i)$, $(y_i)$ converging to infinity are declared to be equivalent if $(x_i,y_i)_o\to\infty$ as $i\to\infty$. The \emph{(Gromov) boundary} of $X$, denoted $\partial X$, is the set of equivalence classes of sequences converging to infinity. We also say that a sequence $(x_i)$ converging to infinity \emph{converges to} its equivalence class representing a point in $\partial X$. Since a change of basepoint modifies the Gromov product by a uniformly bounded amount, all these notions are actually independent of the chosen $o\in X$. 

When $X$ is a $\mathrm{CAT}(-1)$ space, that is, a metric space of actual negative curvature, the Gromov product can be extended from $X$ to the bordification $X\cup \bd X$ by setting $(\xi,\xi')_o:=\lim\; (x_i,x'_i)_o$ for $x_i\to\xi$ and $x'_i\to\xi'$, respectively $(x,\xi)_o:=\lim\; (x,x_i)_o$ for $x_i\to\xi$. But for a general hyperbolic space $X$, such a process is not well-defined in general: the limit may not exist, and even if it does, it might depend on the choice of a representing sequence. The traditional procedure of getting around this problem runs as follows. The Gromov product on $\bd X\times \bd X$ is defined by setting $(\xi,\xi')_o:=\inf \{\liminf\; (x_i,x'_i)_o: \; x_i\to\xi,\; x'_i\to\xi'\}$. This is a finite quantity, except for $(\xi,\xi)_o=\infty$. Similarly, the Gromov product on $X\times \bd X$ is defined by $(x,\xi)_o:=\inf \{\liminf\; (x,x_i)_o: \; x_i\to\xi\}$. It turns out that the other reasonable choices, namely taking sup instead of inf, or $\limsup$ instead of $\liminf$, are all uniformly close to each other. Although rather ad-hoc, such an extension of the Gromov product has at least the benefit of being canonical at the topological level: it defines a topology on the bordification $X\cup \bd X$, for which convergence at infinity takes on an actual meaning. If $X$ is proper then both $X\cup \bd X$ and $\bd X$ are compact.

It turns out that the topology on $\bd X$ can be metrized. Namely, for small enough $\e>0$ there exists a metric $d_\e$ on $\partial X$, such that $d_\e$ is within constant positive multiples of $\exp(-\e\:(\cdot,\cdot)_o)$. When $X$ is a $\mathrm{CAT}(-1)$ space, $\exp(-\:(\cdot,\cdot)_o)$ is already a metric on $\bd X$, but this is not true in general \cite[Lemma 7]{Min2}. To summarize, the fuzziness of the Gromov product at the boundary is resolved by a non-canonical choice, which in turn leads to a quasified metric structure on the boundary.

\subsection{Roughly geodesic spaces}
The classical theory of hyperbolic spaces works under the assumption that the spaces are geodesic. But the theory is just as solid for roughly geodesic spaces. These are defined as follows.

\begin{defn}
A metric space $X$ is \emph{roughly geodesic} if there is a constant $C\geq 0$ such that the following holds: for every $x,y\in X$, there is a (not necessarily continuous) map $\gamma: [a,b]\to X$ such that $\gamma(a)=x$, $\gamma(b)=y$, and $|s-s'|-C\leq |\gamma(s),\gamma(s')|\leq |s-s'|+C$ for all $s,s'\in [a,b]$. 
\end{defn}  

The viewpoint that roughly geodesic spaces provide a natural relaxation of the geodesic assumption for hyperbolic spaces was first advocated by Bonk and Schramm in \cite{BS}. We remark that the notion of weakly geodesic space used by Lafforgue \cite{Laf} and Mineyev - Yu \cite{MY}, although weaker in general, is equivalent to the notion of roughly geodesic space within the hyperbolic world.

Hyperbolicity is a quasi-isometry invariant for roughly geodesic spaces, and the Schwarz - Milnor lemma holds in the roughly geodesic context. So we may say that a group is \emph{hyperbolic} if it admits a geometric action (that is, isometric, proper, and cocompact) on a roughly geodesic, proper hyperbolic space.

Throughout this paper, hyperbolic groups are assumed to be non-elementary.

\section{Good hyperbolic spaces}\label{sec: good spaces}
Let us recall the definition we made in the Introduction, this time in a quantified form.

\begin{defn} We say that a hyperbolic space $X$ is \emph{$\e$-good}, where $\e>0$, if the following two properties hold for each basepoint $o\in X$:
\begin{itemize}
\item the Gromov product $(\cdot,\cdot)_o$ on $X$ extends continuously to the bordification $X\cup \bd X$, and 
\item $\exp(-\e \:(\cdot,\cdot)_o)$ is a metric on the boundary $\bd X$.
\end{itemize}
\end{defn}

Let $X$ be a good hyperbolic space. The continuity of the Gromov product implies that we have a well-defined notion of Busemann function
\begin{align*}
\beta(y,z;\xi)=2(\xi, z)_y-|y,z|=\lim_{x\to \xi} \big(|x,y|-|x,z|\big)\qquad (y,z\in X, \xi\in \bd X).\end{align*}

Incidentally, it can be checked that the Busemann boundary of $X$ coincides with the Gromov boundary $\bd X$ when $X$ is proper. (We owe this remark to Peter Ha\"issinsky.) This topological consequence can be very useful - see, for instance, \cite{Bjo} and \cite{HM} - but we will not need it here.

For the purposes of this paper, the main feature of a good hyperbolic space is its \emph{conformal structure} on the boundary - as opposed to mere quasi-conformality for general hyperbolic spaces. This is discussed in distinct but related ways in \cite[Sections 7 and 8]{Min2}, \cite[Section 5.1]{HaiB}, \cite[Section 4]{Nic}. See also \cite[Sections 2.6 and 2.7]{Bou} for the $\mathrm{CAT}(-1)$ case. Let $g$ be an isometry of $X$. For all $x,x'\in X$ we have that
\begin{align*}
-2(gx,gx')_o=|x,o|-|x,g^{-1}o|+|x',o|-|x',g^{-1}o|-2(x,x')_o.
\end{align*}
Letting $x\to \xi \in \bd X$ respectively $x'\to \xi' \in \bd X$, and using the notation $d_{\e,o}$ for the visual metric $\exp(-\e \:(\cdot,\cdot)_o)$ on $\bd X$, we obtain the relation
\begin{align*}
d^{\:2}_{\e,o}(g\xi,g\xi')=\exp\big(\e\:\beta(o,g^{-1}o;\xi)\big)\exp\big(\e\:\beta(o,g^{-1}o;\xi')\big)\:d^{\:2}_{\e,o}(\xi,\xi')
\end{align*}
for all $\xi,\xi'\in \bd X$. This is a strong form of \emph{metric conformality} for the action of $\mathrm{Isom}(X)$ on $\bd X$. Minor additional assumptions on $\bd X$ \cite[Lemma 8]{Nic} then guarantee the following \emph{measure-theoretic conformality}:
\begin{align*}
\frac{\mathrm{d}g^*\mu_o}{\mathrm{d}\mu_o}(\xi) =\exp\big(\e D_\e\:\beta(o,g^{-1}o;\xi)\big)
\end{align*}
where $D_\e$ is the Hausdorff dimension of $d_{\e,o}$ (independent of the basepoint $o$), and $\mu_o$ is the Hausdorff measure defined by $d_{\e,o}$ (independent of the visual parameter $\e$). 

For applications, the most interesting case is when $X$ is a proper, roughly geodesic, good hyperbolic space endowed with a geometric action of a (hyperbolic) group $\G$. By the above discussion, $\G$ acts conformally on $\bd X$.

\section{Strongly hyperbolic spaces}\label{sec: super}

\begin{defn} We say that a metric space $X$ is \emph{strongly hyperbolic} with parameter $\e>0$ if
\begin{align*}
\exp(-\e\: (x,y)_o)\leq \exp(- \e\: (x,z)_o)+\exp(-\e\:(z,y)_o).
\end{align*}
for all $x,y,z,o\in X$; equivalently, the four-point condition 
\begin{align*}
\exp\big(\tfrac{1}{2}\e\:(|x,y|+|z,t|)\big)\leq \exp\big(\tfrac{1}{2}\e\:(|x,z|+|y,t|)\big)+\exp \big(\tfrac{1}{2}\e\:(|z,y|+|x,t|)\big)
\end{align*}
holds for all $x,y,z,t\in X$.
\end{defn}
One motivation for considering this notion is that it grants the enhancements of hyperbolicity that we have mentioned. Here is the quantified version of Theorem~\ref{thm: A}.

\begin{thm}\label{properties}
Let $X$ be a strongly hyperbolic space, with parameter $\e$. Then $X$ is an $\e$-good, $(\log 2)/\e$-hyperbolic space. Furthermore, $X$ is strongly bolic provided $X$ is roughly geodesic. 
\end{thm}

The proof is straightforward. We include the details.

\begin{proof}
From 
\begin{align*}
\exp(-\e\: (x,y)_o)&\leq 2\max\big\{\exp(- \e\: (x,z)_o),\exp(-\e\:(z,y)_o)\big\}\\
&=2\exp\big(- \e\: \min\{(x,z)_o,(z,y)_o\}\big)
\end{align*}
we see that $X$ is $(\log 2)/\e$-hyperbolic.

Next we check that the Gromov product $(\cdot,\cdot)_o$ extends continuously to the bordification $\overline{X}=X\cup \bd X$; it will then follow that $\exp(-\e \:(\cdot,\cdot)_o)$ is a metric on $\bd X$. Let $z\in X$ and $\xi\in \bd X$. If $(x_i)$ is a sequence in $X$ converging to $\xi$, then 
\begin{align*}
\big|\exp(-\e\: (z,x_i)_o)-\exp(- \e\: (z,x_j)_o)\big|\leq\exp(-\e\:(x_i,x_j)_o)\to 0 \quad \textrm{ as } i,j\to \infty 
\end{align*}
so we may put $(z,\xi)_o:=\lim\: (z,x_i)_o$. To see that this is well-defined, let $(y_i)$ be another sequence in $X$ converging to $\xi$; then 
\begin{align*}
\big|\exp(-\e\: (z,x_i)_o)-\exp(- \e\: (z,y_i)_o)\big|\leq\exp(-\e\:(x_i,y_i)_o)\to 0 \quad \textrm{ as } i\to \infty. 
\end{align*}
Thus, we have a continuous extension of $(\cdot,\cdot)_o$ to $X\times \overline{X}$. Similar arguments lead to a continuous extension of $(\cdot,\cdot)_o$ to $\overline{X}\times \overline{X}$. 

For hyperbolic spaces which are roughly geodesic, strong bolicity in the sense of Lafforgue \cite{Laf} amounts to the following: for every $\eta, r>0$ there exists $R>0$ such that $|x,y|+|z,t|\leq r$ and $|x,z|+|y,t|\geq R$ imply $|x,t|+|y,z|\leq |x,z|+|y,t|+\eta$. Since
\begin{align*}
\exp\big(\tfrac{1}{2}\e\:(|x,t|+|y,z|-|x,z|-|y,t|)\big)\leq 1+ \exp\big(\tfrac{1}{2}\e\:(|x,y|+|z,t|-|x,z|-|y,t|)\big)
\end{align*}
it suffices to choose $R>0$ such that $1+\exp \tfrac{1}{2}\e(r-R)\leq \exp\tfrac{1}{2}\e \eta$.
\end{proof}

The next two sections are devoted to interesting examples of strongly hyperbolic spaces.

\section{Strong hyperbolicity for $\mathrm{CAT}(-1)$ spaces}
It is known that $\mathrm{CAT}(-1)$ spaces are strongly bolic and good hyperbolic spaces. The next result strengthens and unifies these facts.

\begin{thm}\label{thm: Cat minus one} 
$\mathrm{CAT}(-1)$ spaces are strongly hyperbolic with parameter $1$.
\end{thm}

Let us recall the four-point characterization of the $\mathrm{CAT}(-1)$ property for a geodesic space $X$. Informally, every quadrilateral in $X$ has a comparison quadrilateral in the hyperbolic plane $\Hyp^2$ whose diagonals are at least as long as the diagonals of the original quadrilateral in $X$. Formally, for every choice of points $x_1,x_2,x_3,x_4\in X$, there are points $\tilde{x}_1,\tilde{x}_2,\tilde{x}_3,\tilde{x}_4\in \Hyp^2$ such that 
\begin{align*}
|x_1,x_2|=|\tilde{x}_1,\tilde{x}_2|,\quad |x_2,x_3|=|\tilde{x}_2,\tilde{x}_3|,\quad |x_3,x_4|=|\tilde{x}_3,\tilde{x}_4|,\quad |x_4,x_1|=|\tilde{x}_4,\tilde{x}_1|
\end{align*}
while $|x_1,x_3|\leq|\tilde{x}_1,\tilde{x}_3|$ and $|x_2,x_4|\leq|\tilde{x}_2,\tilde{x}_4|$. 

It therefore suffices to prove Theorem~\ref{thm: Cat minus one} for $\Hyp^2$, the model for metric spaces of curvature at most $-1$.

\begin{prop}\label{hyp plane} 
The hyperbolic plane $\Hyp^2$ is strongly hyperbolic with parameter $1$. 
\end{prop}

\begin{proof}
We use the upper half-space model $\{w\in \C: \im w>0\}$, with distance given by
\begin{align*}
\cosh |w_1,w_2|=1+\frac{|w_1-w_2|^2}{2(\im w_1)(\im w_2)},
\end{align*}
that is,
\begin{align*}
|w_1,w_2|=2\log \frac{|w_1-w_2|+|w_1-\overline{w}_2|}{2\sqrt{\im w_1}\sqrt{\im w_2}}.
\end{align*}
Using the notation $\|w_1,w_2\|=|w_1-w_2|+|w_1-\overline{w}_2|$, the four-point formulation of strong hyperbolicity with parameter $1$ becomes simply
\begin{align*}
\|w_1,w_3\|\|w_2,w_4\|\leq \|w_1,w_2\|\|w_3,w_4\|+\|w_1,w_4\|\|w_2,w_3\|.
\end{align*}
This brings to mind Ptolemy's inequality in the plane: 
\begin{align*}
|w_1-w_3|\: |w_2-w_4|\leq |w_1-w_2|\:|w_3-w_4|+|w_1-w_4|\:|w_2-w_3|.
\end{align*} 
And indeed, several applications of Ptolemy's inequality yield the desired inequality. Here is a more conceptual formulation. Let $X$ be a metric space which is ptolemaic, i.e., $|x_1,x_3|\: |x_2,x_4|\leq |x_1,x_2|\: |x_3,x_4|+|x_1,x_4|\: |x_2,x_3|$ for all $x_1,x_2,x_3,x_4\in X$, and let $G$ be a finite group of isometries of $X$. (In our case, $X$ is the Euclidean plane and $G$ is the cyclic group of order $2$ generated by the reflection in the $x$-axis.) Then, on the quotient $X/G$, the formula 
\begin{align*}
\big\|[x],[y]\big\|=\frac{1}{|G|}\sum_{g\in G} |x,gy| \qquad \textrm{ for } [x]\neq [y]
\end{align*}
defines a ptolemaic metric. Both the triangle inequality and Ptolemy's inequality are established by averaging over the $G$-orbits of all but one of the variables.
\end{proof}

\begin{rem} Let $X$ be strongly hyperbolic with parameter $\e$, and fix a basepoint $o\in X$. On $X$, let us write $d_\e(x,y)$ for $\exp(-\e\: (x,y)_o)$. The strong hyperbolicity assumption means that $d_\e$ satisfies the triangle inequality; however, $d_\e$ fails to be an actual metric on $X$ because $d_\e(x,x)>0$ for each $x\in X$. Now observe that $d_\e$ fulfills Ptolemy's inequality 
\begin{align*}
d_\e(x_1,x_3)\: d_\e(x_2,x_4)\leq d_\e(x_1,x_2)\: d_\e(x_3,x_4)+d_\e(x_1,x_4)\: d_\e(x_2,x_3),
\end{align*}
for this is simply another way of writing the four-point formulation of strong hyperbolicity with parameter $1$. Consequently, at infinity, the visual \emph{metric} $d_\e=\exp(-\e \:(\cdot,\cdot)_o)$ on $\bd X$ is ptolemaic.

Foertsch and Schroeder have shown \cite[Theorems 1 and 16]{FS} that, for a $\mathrm{CAT}(-1)$ space $X$, the visual metric $d_\e=\exp(-(\cdot,\cdot)_o)$ on the boundary $\bd X$ is ptolemaic. In light of Theorem~\ref{thm: Cat minus one}, the Foertsch - Schroeder result is a particular case of the fact pointed out above.
\end{rem}

As another consequence of Proposition~\ref{hyp plane}, we find the best constant of hyperbolicity (in the sense of Definition~\ref{defn: gromov hyperbolic}) for $\Hyp^2$. This answers a folklore question.  

\begin{cor}\label{cor: optimal}
$\Hyp^2$ is $\log 2$-hyperbolic, and this is optimal.
\end{cor}

\begin{proof}
That $\Hyp^2$ is $\log 2$-hyperbolic follows from Theorem~\ref{properties}. To see that this is best possible, let $\Hyp^2$ be $\delta$-hyperbolic for some $\delta\geq 0$. Sticking to the upper half-space model, the hyperbolic inequality at infinity says that $(a,b)_i\geq \min\{(a,c)_i, (c,b)_i\}-\delta$ whenever $a,b,c\in \R$ are distinct. At infinity, the Gromov product with respect to $i$ is given by
\begin{align*}
(a,b)_i=\log \frac{|a-i|\: |b-i|}{|a-b|}
\end{align*} 
for distinct $a,b\in \R$. For $a> 0$, $b=-a$, $c=0$, we are quickly led to $|a-i|\geq 2\exp(-\delta)$. Hence $\exp \delta\geq 2$, by letting $a$ tend to $0$.
\end{proof}
For comparison, we recall the well-known fact that $\Hyp^2$ is $\log(1+\sqrt{2})$-hyperbolic for the Rips definition, in terms of thin triangles, and this is optimal.


\section{Strong hyperbolicity of the Green metric on a hyperbolic group}
Another example of a strongly hyperbolic metric is Mineyev's `hat' metric from \cite{Min1, Min2}. In fact, our notion of strong hyperbolicity is inspired by the proof of \cite[Theorem 5]{Min2}. Let us recall that Mineyev's hat metric is a (type of) metric defined on any hyperbolic group $\G$, with the following properties: it is $\G$-invariant, quasi-isometric to any word-metric on $\G$, roughly geodesic, and - in our terminology - strongly hyperbolic. Of course, the latter property is the crucial one.

Mineyev's hat metric is a very useful tool, but it has the conceptual downside of being custom-made. One would like a `natural' type of metric doing the same job - that of making hyperbolic groups strongly hyperbolic. We show that this is the case for the Green metric coming from a random walk.

\begin{thm}\label{supergreen}
Let $\G$ be a hyperbolic group. Then the Green metric defined by a symmetric and finitely supported random walk on $\G$ is strongly hyperbolic.
\end{thm}

We recall the essential facts on the Green metric, and we refer to \cite{BHM} and \cite{Hai} for more details. A probability measure $\mu$ on $\Gamma$ defines a random walk on $\Gamma$ with transition probabilities $p(x,y)=\mu(x^{-1}y)$. The probability measure $\mu$, or the corresponding random walk, is \emph{symmetric} if $\mu(x)=\mu(x^{-1})$ for every $x\in \G$, and \emph{finitely supported} if the support of $\mu$ is a finite generating set for $\Gamma$. 

Let $F(x,y)$ be the probability that the random walk started at $x$ ever hits $y$. The \emph{Green metric} on $\Gamma$, first introduced by Blach\'ere and Brofferio \cite{BB}, is given by
\begin{align*}
	|x,y|_G = -\log F(x,y).
\end{align*}
The name has to do with the \emph{Green function} $G(x,y)=\sum_{n\geq 0}\: \mu^n(x^{-1}y)$, where $\mu^n$ denotes the $n$-th convolution power of $\mu$, counting the expected number of visits of $y$ starting from $x$. It turns out that $F$ and $G$ are proportional: $G(e,e)F(x,y)=G(x,y)$.

The Green metric is $\G$-invariant, quasi-isometric to any word-metric on $\G$, and roughly geodesic. The first two properties are generic, in the sense that they only use the non-amenability of $\G$. The latter property follows from a result of Ancona \cite{Anc} saying that geodesics for the word-metric are rough geodesics for the Green metric, and it relies on the hyperbolicity $\G$. It follows that the Green metric is a hyperbolic metric, but our aim is to show more.

We remark that strong bolicity for the Green metric was first shown by Ha\"issinsky and Mathieu in \cite{HM}. For their short proof it suffices to know Ancona's result that the Martin boundary of the random walk coincides with the Gromov boundary. The fact that the Green metric on a hyperbolic group is a good hyperbolic metric answers a question raised in \cite{Nic}, and for the purposes of \cite{Nic} it means that the Green metric can be used in place of Mineyev's hat metric from \cite{Min1, Min2}. 

\subsection{Proof of Theorem~\ref{supergreen} }The next lemma spells out what else is needed, on top of hyperbolicity, in order to have strong hyperbolicity.

\begin{lem}\label{lemma EB} A hyperbolic space $X$ is strongly hyperbolic if and only if the following holds: 
\begin{itemize}
\item[\textsc{(EB)}] there exist $L,\lambda>0$ and $R_0>0$ such that $\big||x,y|+|z,t|-|x,t|-|z,y|\big|\leq L\exp(-\lambda R)$ whenever $x,y,z,t\in X$ satisfy $|x,y|+|z,t|-|x,z|-|y,t|\geq R\geq R_0$.
\end{itemize}
\end{lem}

The \textsc{(EB)} condition may be interpreted as an exponentially strong bolicity, since merely knowing that $|x,y|+|z,t|-|x,t|-|z,y|\to 0$ as $|x,y|+|z,t|-|x,z|-|y,t|\to \infty$ already implies strong bolicity (for hyperbolic spaces which are roughly geodesic).

\begin{proof}
For the sake of notational simplicity, let us write 
\begin{align*}
A:=\tfrac{1}{2}\big(|x,y|+|z,t|-|x,z|-|y,t|\big), \qquad B:=\tfrac{1}{2}\big(|x,y|+|z,t|-|x,t|-|z,y|\big).
\end{align*}

\smallskip
($\Leftarrow$) The four-point condition for strong hyperbolicity with parameter $\e$ amounts to showing that, for some $\e>0$, we have $1 \leq \exp(-\e A)+\exp(-\e B)$ . We may assume that $A,B\geq 0$, and $A\geq B$. We rewrite the desired inequality as $\e B\leq -\log (1-\exp(-\e A))$. As $a\leq -\log(1-a)$ for $a\in [0,1)$, it suffices to obtain $\e B\leq \exp(-\e A)$. When $A\geq R_0/2$, this can be achieved for $\e\leq \min\{2\lambda,2/L\}$ by using \textsc{(EB)}. When $A<R_0/2$, we use the assumption that $X$ is hyperbolic: $B=\min\{A,B\}\leq \delta$, a hyperbolicity constant of $X$. So it suffices to have $\e\delta\leq \exp(-\e R_0/2)$, and this can be achieved for, say, $\e\leq \min\{1/(2\delta), (2\log 2)/R_0\}$.

\smallskip
($\Rightarrow$) Strong hyperbolicity with parameter $\e$ yields $1 + \exp(-\e A)\geq \exp(-\e B)\geq 1-\exp(-\e A)$. When $A\geq 0$, this means that $-\log (1 + \exp(-\e A))\leq \e B\leq -\log(1-\exp(-\e A))$. Now $\log(1+a)\leq 2a$ and $ -\log(1-a)\leq 2a$ for small enough $a\in [0,1)$. Thus, there is $A_0>0$ such that $|\e B|\leq 2\exp(-\e A)$ whenever $A\geq A_0$. We deduce \textsc{(EB)} with $L=4/\e$, $\lambda=\e/2$, $R_0=2A_0$.
\end{proof}

The following is the key estimate for the proof of Theorem~\ref{supergreen}. It is a particular instance of Gou\"ezel's strong uniform Ancona inequalities \cite[Definition 2.8, Theorem 2.9]{Gou}. It can also be deduced from an adaptation of \cite[Lemma 3.2]{INO}, and that was, in fact, our original approach before becoming aware of the very recent \cite{Gou}. 

\begin{lem}\label{key est}
There exist $L,\lambda>0$ and $R_0>0$ such that
\begin{align*}
\bigg|\frac{G(x,y)\:G(z,t)}{G(x,t)\: G(z,y)}-1\bigg|\leq L\exp(-\lambda R)
\end{align*} 
whenever the distance between geodesics $[x,z]$ and $[y,t]$ is at least $R\geq R_0$.
\end{lem}

We emphasize that, in the above lemma, geodesics are taken with respect to the word-metric on $\G$. Let us derive condition \textsc{(EB)} for the Green metric on $\G$. On the one hand, as $|\log a|$ is asymptotic to $|a - 1|$ for $a\to 1$, the conclusion of Lemma~\ref{key est} can be stated as 
\begin{align*}
\big||x,y|_G+|z,t|_G-|x,t|_G-|z,y|_G\big|=\bigg|\log \frac{G(x,y)\:G(z,t)}{G(x,t)\: G(z,y)}\bigg|\leq L\exp(-\lambda R),
\end{align*}
up to increasing $R_0$ if necessary. (Compare, at this point, with \cite[Theorem 32d]{Min1}.)

On the other hand, we claim that there is some spatial constant $C>0$ such that
\begin{align*}
C\:\mathrm{dist}\big([x,z], [y,t]\big)+C\geq |x,y|_G+|z,t|_G-|x,z|_G-|y,t|_G
\end{align*}
where, once again, the left hand side is in terms of the word metric. Let $p$ and $q$ be points on $[x,z]$, respectively $[y,t]$, with $|p,q|=\mathrm{dist}\big([x,z], [y,t]\big)$. (See the figure below.) We have $|p,q|_G\leq C_1|p,q|+C_1$ for some spatial constant $C_1>0$, by the quasi-isometry of the word and the Green metrics. Also, there is a spatial constant $C_2>0$ such that
\begin{align*}
|x,p|_G+|p,z|_G\leq |x,z|_G+C_2, \qquad |y,q|_G+|q,t|_G\leq |y,t|_G+C_2
\end{align*}
since geodesics for the word metric are rough geodesics for the Green metric. It follows that
\begin{align*}
|x,y|_G+|z,t|_G&\leq \big(|x,p|_G+|p,q|_G+|q,y|_G\big)+\big(|z,p|_G+|p,q|_G+|q,t|_G\big)\\
&\leq |x,z|_G+|y,t|_G+ 2|p,q|_G+2C_2\\
&\leq |x,z|_G+|y,t|_G+2C_1 |p,q|+ (2C_1+2C_2)
\end{align*}
and the claim is proved. This completes the verification of condition \textsc{(EB)} for the Green metric.
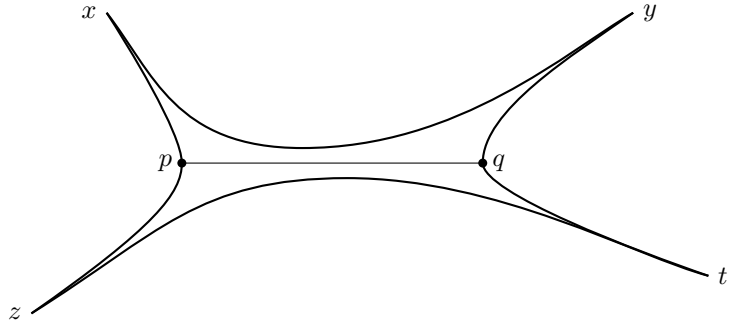
\begin{figure}[h!]
  \begin{tikzpicture}[thick,>=stealth,scale=1]
    \coordinate [label=left:$x$] (X) at (1,4);
    \coordinate [label=left:$z$] (Z) at (0,0);
    \coordinate [label=right:$t$] (T) at (9,0.5);
    \coordinate [label=right:$y$] (Y) at (8,4);
    \coordinate [label=left:$p$] (P) at (2,2);
    \coordinate [label=right:$q$] (Q) at (6,2);

    \draw (X) .. controls +(-57:.5) and +(up:.4) .. (P) node [pos=.4,left] {$$} .. controls +(down:.4) and +(35:2) .. (Z);
    \draw (Z) .. controls +(32:2) and +(left:2) .. ($(P)!0.55!(Q)+(down:.2)$) .. controls +(right:2) and +(165:1) .. (T) node [pos=0.53,below] {$$};
    \draw (T) .. controls +(161:1) and +(down:0.4) .. (Q) node [pos=.5,above] {$$} .. controls +(up:.7) and +(-145:1) .. (Y);
    \draw (Y) .. controls +(-150:1) and +(right:2.1) .. ($(P)!0.4!(Q)+(up:.2)$) .. controls +(left:1.7) and +(-52:1) .. (X) node [pos=.7,right] {$$};
    \draw [thin] (P) -- (Q);

    \fill (P) circle (1.7pt);
    \fill (Q) circle (1.7pt);

  \end{tikzpicture}
\end{figure}

\subsection{Harmonic measure} We will now explain a consequence of Theorem~\ref{supergreen} which concerns the harmonic measure on the boundary. In what follows, we write $(\cdot,\cdot)_G$ for the Gromov product with respect to the Green metric, and having the identity element $e\in \G$ as the basepoint.

The random walk $(Z_n)$ started at $e\in \G$ converges almost surely to a boundary point $Z_\infty$. The \emph{harmonic measure} $\nu$ is the probability measure on $\bd \G$ defined by the condition that $\nu(A)$ is the probability that $Z_\infty$ is in $A\subseteq \bd \G$. By \cite[Section 3.4]{BHM}, see also the proof of Proposition 3.6 in \cite{Hai}, one knows that $\nu$ is a $\G$-conformal measure:
\begin{align}\label{RNd}
\frac{\mathrm{d}g^*\nu}{\mathrm{d}\nu}(\xi) =\exp\beta_G(e,g^{-1};\xi)
\end{align}
where $\beta_G$ is the Busemann function with respect to the Green metric. Recall that
\begin{align*}
\beta_G(e,g;\xi)=2(g,\xi)_G-|g|_G=\lim_{x\to \xi} \big(|x,e|_G-|x,g|_G\big).	
\end{align*}
In \cite{BHM} and \cite{Hai}, the existence of the above limit follows from the identification of the Gromov boundary with the Martin boundary of the random walk. Herein, we see it as a manifestation of the good hyperbolicity of the Green metric.

On the other hand, good hyperbolicity of the Green metric means that the boundary visual structure induced by the Green metric is conformal, as explained in Section~\ref{sec: good spaces}. We endow $\bd \G$ with a visual metric $d_\e=\exp(-\e\: (\cdot,\cdot)_G)$ for small enough $\e>0$, and we let $\mu$ denote the normalized Hausdorff measure (independently of $\e$). Then $\mu$ is $\G$-conformal, with
\begin{align}\label{RNHaus}
\frac{\mathrm{d}g^*\mu}{\mathrm{d}\mu}(\xi) =\exp\beta_G(e,g^{-1};\xi)
\end{align}
since the Hausdorff dimension of $d_\e$ is $1/\e$ by \cite[Theorem 1.1]{BHM}. 

Now \cite[Theorem 2.7]{BHM} says that any two $\G$-quasi-conformal measures on $\bd\G$ are equivalent ergodic measures, with Radon - Nikodym derivative bounded above and below by positive constants. Equations \eqref{RNd} and \eqref{RNHaus} imply that the Radon - Nikodym derivative $\mathrm{d}\nu/\mathrm{d}\mu$ is $\G$-invariant, hence constant a.e. by ergodicity. We conclude:

\begin{cor}\label{cor: two equal measures}
Let $\G$ be a hyperbolic group. Then the harmonic measure on $\bd \G$ defined by a symmetric and finitely supported random walk on $\G$ equals the Hausdorff probability measure defined by any Green visual metric $d_\e=\exp(-\e\: (\cdot,\cdot)_G)$.
\end{cor}

\begin{ack}
We thank Peter Ha\"issinsky for useful comments. B.N. acknowledges support from the Alexander von Humboldt Foundation.
\end{ack}



\medskip
\begin{thebibliography}{ww}
\bibitem{Anc} A. Ancona: \emph{Th\'eorie du potentiel sur les graphes et les vari\'et\'es}, in \emph{\'Ecole d'\'et\'e de Probabilit\'es de Saint-Flour XVIII--1988}, Lecture Notes in Math. 1427, Springer 1990, 1--112

\bibitem{Bjo} M. Bj\"orklund: \emph{Central limit theorems for Gromov hyperbolic groups},
J. Theoret. Probab. 23 (2010), no. 3, 871--887


\bibitem{BB} S. Blach\`ere, S. Brofferio: \emph{Internal diffusion limited aggregation on discrete groups having exponential growth},  Probab. Theory Related Fields 137 (2007), no. 3-4, 323--343

\bibitem{BHM} S. Blach\`ere, P. Ha\"issinsky, P. Mathieu: \emph{Harmonic measures versus quasiconformal measures for hyperbolic groups}, Ann. Sci. \'Ec. Norm. Sup\'er. (4) 44 (2011), 683--721

\bibitem{BS} M. Bonk, O. Schramm: \emph{Embeddings of Gromov hyperbolic spaces}, Geom. Funct. Anal. 10 (2000), no. 2, 266--306

\bibitem{Bou} M. Bourdon: \emph{Structure conforme au bord et flot g\'eod\'esique d'un ${\rm CAT}(-1)$-espace}, Enseign. Math. (2) 41 (1995), no. 1-2, 63--102

\bibitem{FS} T. Foertsch, V. Schroeder: \emph{Hyperbolicity, ${\rm CAT}(-1)$-spaces and the Ptolemy inequality}, Math. Ann. 350 (2011), no. 2, 339--356

\bibitem{Gou} S. Gou\"ezel: \emph{Local limit theorem for symmetric random walks in Gromov-hyperbolic groups}, J. Amer. Math. Soc. 27 (2014), no. 3, 893--928

\bibitem{Gro} M. Gromov: \emph{Hyperbolic groups}, in \emph{Essays in group theory}, Publ. MSRI 8, Springer 1987, 75--263

\bibitem{HaiB} P. Ha\"issinsky: \emph{G\'eom\'etrie quasiconforme, analyse au bord des espaces m\'etriques hyperboliques et rigidit\'es (d'apr\`es Mostow, Pansu, Bourdon, Pajot, Bonk, KleinerÉ)}, S\'eminaire Bourbaki 2007/2008, Exp. No. 993, Ast\'erisque 326 (2009), 321--362

\bibitem{Hai} P. Ha\"issinsky: \emph{Marches al\'eatoires sur les groupes hyperboliques}, in \emph{G\'eom\'etrie ergodique}, F. Dal'Bo ed., Monographie de l'Enseignement Math\'ematique 43 (2013), 199--265

\bibitem{HM} P. Ha\"issinsky, P. Mathieu: \emph{La conjecture de Baum-Connes pour les groupes hyperboliques par les marches al\'eatoires}, preprint 2011

\bibitem{INO} M. Izumi, S. Neshveyev, R. Okayasu: \emph{The ratio set of the harmonic measure of a random walk on a hyperbolic group}, Israel J. Math. 163 (2008), 285--316

\bibitem{Laf} V. Lafforgue: \emph{K-th\'eorie bivariante pour les alg\`ebres de Banach et conjecture de Baum-Connes}, Invent. Math. 149 (2002), no. 1, 1--95

\bibitem{Min1} I. Mineyev: \emph{Flows and joins of metric spaces}, Geom. Topol. 9 (2005), 403--482

\bibitem{Min2} I. Mineyev: \emph{Metric conformal structures and hyperbolic dimension}, Conform. Geom. Dyn. 11 (2007), 137--163

\bibitem{MY} I. Mineyev, G. Yu: \emph{The Baum-Connes conjecture for hyperbolic groups}, Invent. Math. 149 (2002),  no. 1, 97--122

\bibitem{Nic} B. Nica: \emph{Proper isometric actions of hyperbolic groups on $L^p$-spaces}, Compos. Math. 149 (2013), no. 5, 773--792


\end{thebibliography}
\end{document}